\documentclass[fleqn,final,3p,11pt]{elsarticle}

\usepackage[font=small,format=plain,labelfont={bf,up},labelsep=space]{caption}

\DeclareCaptionLabelFormat{table}{TABLE~#2}
\makeatletter
\def\ps@pprintTitle{%
 \let\@oddhead\@empty
 \let\@evenhead\@empty
 \def\@oddfoot{\centerline{\thepage}}%
 \let\@evenfoot\@oddfoot}
\makeatother
\usepackage{geometry}                
\usepackage{lipsum}    
\usepackage{graphicx}
\usepackage{amssymb}
\usepackage{epstopdf}
\usepackage{amsmath}
\usepackage{lineno,hyperref}
\usepackage{amsmath}
\usepackage{amsthm}
\usepackage{amsfonts}
\usepackage{float}
\usepackage{algorithm} 
\usepackage{algpseudocode} 
\usepackage{bbm}
\usepackage{bm}
\usepackage{graphicx}
\usepackage{caption}
\usepackage{subcaption}
\usepackage{epstopdf}
\usepackage{booktabs}
\usepackage[svgnames,table]{xcolor}

\newtheorem{theorem}{Theorem}

\newtheorem{proposition}{\textbf{Proposition}}[section]

\newtheorem{assumption}[proposition]{\textbf{Assumption}}
\newtheorem{lemma}[proposition]{\textbf{Lemma}}

\theoremstyle{remark}
\newtheorem{remark}[proposition]{Remark}                                            
\begin{document}
\begin{frontmatter}
\title{ Convergence properties of a Gauss-Newton
data-assimilation method}
\author[add1]{Nazanin Abedini}
\ead{n.abedini@vu.nl}
\author[add1]{Svetlana Dubinkina}
\ead{s.b.dubinkina@vu.nl}
\date{} 
\address[add1]{VU Amsterdam, Department of Mathematics, De Boelelaan 1111, 1081 HV Amsterdam, The Netherlands}


\begin{abstract}
Four-dimensional weak-constraint variational data assimilation estimates a state given partial noisy observations and dynamical model by minimizing 
a cost function that takes into account both discrepancy between the state and observations and model error over time. It can be formulated as a Gauss-Newton iteration of an associated least-squares problem. 
In this paper, we introduce a parameter in front of the observation mismatch and show analytically that this parameter is crucial either for convergence to the true solution when observations are noise-free or for boundness of the error when 
observations are noisy with bounded observation noise. We also consider joint state-parameter estimation. We illustrated theoretical results with numerical experiments using the Lorenz 63 and Lorenz 96 models.
\end{abstract}
\end{frontmatter}

\section{INTRODUCTION}
Data assimilation (DA) estimates a state of a dynamical model given partial noisy measurements (also called observations), e.g.~\cite{Jaz70}. A so-called variational DA minimizes a cost function that is a difference between an estimation and the observations provided that either the estimate is a solution of the dynamical model (strong-constraint variational DA, see e.g. \cite{lewis1985use, talagrand1997assimilation, talagrand1987variational, sasaki1970some}) or the dynamical model equations are satisfied under model error (weak-constraint variational DA, see e.g.~\cite{tr2006accounting, tremolet2007model}). 

Variational DA can be viewed as an approach of combining model and observations based on Tikhonov-Philips regularization, which is used to solve ill-posed inverse problems \cite{johnson2005very}. Inverse problems appear in diverse fields such as geosciences and image reconstruction \cite{lorenc2000met, nichols2003data, engl2000inverse, kaipio2006statistical}. Inverse problems are concerned with seeking a (stationary) solution of a mathematical model given a set of noisy and incomplete observations.  Due to sparsity of observations, the corresponding discrete inverse problem has a highly ill-conditioned coefficient matrix. In order to obtain a stable solution to an ill-posed inverse problem, regularization methods are required. 

Motivated by the weak-constraint variational DA, we consider a DA method that minimizes a cost function under assumption of model error. We are seeking a solution over a time window at once akin to the four-dimensional variational DA (WC4DVar). The main difference to WC4DVar is that the cost function we consider has a parameter in front of the observation mismatch. As we will show, this parameter plays a significant role in the convergence of the DA method and has to be chosen carefully in order to achieve either convergence or boundedness of the error. As in \cite{brett2013accuracy} we use the dynamical model to regularize the least-squares minimization problem. 

There are several research papers investigating error propagation of variational data assimilation. The error is defined as a norm of the difference between the DA estimate and the true trajectory, from which the observations are generated.
In \cite{hayden2011discrete}, the authors considered the Lorenz 63 model and the Navier-Stocks equations and showed that the error converges to zero asymptotically in time for noise-free observations. In \cite{brett2013accuracy,MLPL2013}, it was shown that for a nonlinear dynamical model the error is bounded under assumption of contractivity in the high modes and bounded observation noise. We do not assume contractivity but assume bounds on some operators. The other difference in our analysis compared to the existing analysis of variational DA is that we derive convergence and bounds on the error as iteration goes to infinity and not as the time goes to infinity. The DA method considered in this paper is based on a Gauss-Newton iteration of the least-squares minimization problem, e.g. \cite{gratton2007approximate, cartis2021convergent}, which is was also considered for incremental four-dimensional DA \cite{courtier1994strategy} in \cite{lawless2005investigation, lawless2005approximate}.
 
The paper is organised as follows. In Section \ref{GNsec}, we describe the DA minimization problem and introduce a Gauss-Newton DA iteration to solve it. In Section \ref{erran}, we derive an error convergence result for noise-free observations and an error boundess result for noisy observations with bounded noise.
In Section \ref{param}, we extend the DA method to joint state-parameter estimation.
In Section \ref{sec5}, we illustrate the theoretical results with numercilac experiments using the Lorenz 63 and Lorenz 96 models. Finally, we present our conclusions in Section \ref{conc}.

\section{GAUSS-NEWTON DATA-ASSIMILATION METHOD}\label{GNsec}
Let us consider the following nonlinear dynamical model 
\begin{equation}\label{non_lin}
\frac{du^{\dagger}}{dt} = f(u^{\dagger}), \quad u^{\dagger}(t) \in \mathbb{R}^n, \ t\in[0,T], 
\end{equation}
where $f:\mathbb{R}^n\to\mathbb{R}^n$. Since in many applications the model is defined by the time-discretization, we consider data assimilation in the context of a discrete deterministic model. Let $0=t_0<t_1<\dots<t_{N-1}=T$ be an equidistant partition of $I=[0,T]$ with $t_j=j \Delta t$, then the time-discretization of Equation (\ref{non_lin}) is 
\begin{equation}\label{eq1}
u^{\dagger}_{j+1} = F_j(u^{\dagger}_j), \quad u^{\dagger}_j\in \mathbb{R}^n, \quad j=0,\dots,N-1,
\end{equation}
where $u^{\dagger}_{j+1}=u^{\dagger}(t_{j+1})$ and $F_j$ is twice continuously differentiable $\forall j=0,\cdots,N$. Let us denote $\mathbf{u}^{\dagger}=\{u^{\dagger}_0,\dots,u^{\dagger}_{N}\}$ the true solution of the model and assumed to be unknown. Suppose a sequence of noisy observations  $\mathbf{y} := \{y_{k_0},\dots,y_{k_M}\}$, where $k_0\geq 0$ and $k_M\leq N$, related to $\mathbf{u}^{\dagger}$ is given as
\begin{equation}\label{eq2}
y_{j} = \mathcal{H}_j u^{\dagger}_{j}+\eta_{j}, \quad y_j \in \mathbb{R}^b, \quad j=k_0,\dots,k_M,
\end{equation}
where $\mathcal{H}_j:\mathbb{R}^n\to\mathbb{R}^b$, $b\le n$, is the linear observation operator, and the observation noise $\eta_{j}$, either determenistic or stochastic, is bounded. The goal of data assimilation is to find a state $\mathbf{u}=\{u_0,\cdots,u_N\}$ such that the distance between the state $\mathbf{u}$ and the observation $\mathbf{y}$ is minimized. The weak-constraint variational data assimilation (WC4DVar) \cite{tr2006accounting, talagrand1997assimilation} minimizes the distance to the observations $\mathbf{y}$ under the condition that the estimate is a solution of the dynamical model (Equation \eqref{eq1}) under model error. Namely, WC4DVar solves the following minimization problem
\begin{eqnarray*}
	\min_{\mathbf{u}\in \mathbb{R}^{n N}} \frac{1}{2}\{\|G(\mathbf{u})\|^2 + \|\mathbf{y}-H\mathbf{u}\|^2\}, 
\end{eqnarray*}
where 
\begin{equation}\label{eq4}
 G(\mathbf{u}) = \left(G_0(\mathbf{u})\ G_1(\mathbf{u})\ \cdots\ G_{N-1}(\mathbf{u})
\right)^T, \quad
G_j(\mathbf{u}) = u_{j+1}-F_j(u_j), \quad j=0,\dots,N-1,
\end{equation}
and
\[
H = \left(\mathcal{H}_{k_0}\ \cdots\ \mathcal{H}_{k_M}\right)^T.
\]
In this paper, we abuse the notation of $L^2$-norm in $\mathbb{R}^s$ for different values of dimension $s$.
In WC4DVar the norms are typically weighted by the error covariance matrices. We  consider a similar minimization problem with the sole difference of a parameter $\alpha$ in front of the observation mismatch, namely
\begin{equation}\label{eq5}
	\min_{\mathbf{u}\in \mathbb{R}^{n N}} \frac{1}{2}\{\|G(\mathbf{u})\|^2 + \alpha \|\mathbf{y}-H\mathbf{u}\|^2\}.
\end{equation}
	As we already mentioned in introduction, the parameter $\alpha$ is crucial for the convergence and bound result on error.

We minimize the cost function (Equation (\ref{eq5})) as follows. We start with an initial guess $\mathbf{u}^{(0)} = H^T \mathbf{y} + (I-H^TH) \mathbf{u}^b$, where $\mathbf{y}$ is the observation and $\mathbf{u}^b$ is the background solution, which is usually a forecast state from the previous analysis cycle. Then the method proceeds by the iteration
\begin{equation}\label{eq6}
\mathbf{u}^{(k+1)} = \mathbf{u}^{(k)} -\left(G'(\mathbf{u}^{(k)})^{T} G'(\mathbf{u}^{(k)})+\alpha H^{T}H\right)^{-1} \left(G'(\mathbf{u}^{(k)})^{T}G(\mathbf{u}^{(k)}) + \alpha H^{T}(H\mathbf{u}^{(k)}-\mathbf{y})\right),
\end{equation}
where $k$ denotes the index of the Gauss-Newton's iteration, $G(\mathbf{u}^{(k)})$ defined in Equation (\ref{eq4}), and $G'(\mathbf{u}^{(k)})$ is Jacobian of $G$ which has an $n(N-1)\times nN$ block structure:
\begin{equation}\label{eq7}
G'(\mathbf{u}) = 
\begin{bmatrix}
-F_0^{'}(u_0) & I &\\
 & -F_1^{'}(u_1) & I &\\
  & & \ddots &  \ddots & \\
& & & -F_{N-1}^{'}(u_{N-1}) & I \\
\end{bmatrix}.
\end{equation}
\section{ERROR ANALYSIS}\label{erran}
We define the error between approximation $\mathbf{u}$ of the Gauss-Newton DA method (Equation (\ref{eq6})) and the true solution $\mathbf{u}^{\dagger}$ by
\begin{equation}\label{err_def}
    \mathbf{e}_k := \mathbf{u}^{(k)}-\mathbf{u}^{\dagger}, \quad \forall k=0,1,\dots.
\end{equation}
We show that the norm of error (Equation (\ref{err_def})) either converges to zero or is bounded, dependent whether the observations are noise-free or noisy.

\subsection{\textbf{Noise-free observations}}
We assume that the sequence of observations $\mathbf{y}$ is noise-free, thus $\eta_j=0, \ j=0,\cdots,N$ in Equation (\ref{eq2}). We show that the Gauss-Newton DA method (Equation (\ref{eq6})) produces an accurate state estimation under the vanishing noise assumption. In order to do this, we need the following assumptions on $G'$.
\begin{assumption}\label{ass1}
The Jacobian  $G'(\mathbf{u})$ defined in Equation (\ref{eq7}) is globally Lipschitz continuous 
with the Lipschitz constant denoted by $L_1>0$, namely 
\begin{equation}\label{eqq8}
\|G'(\mathbf{u}) - G'(\mathbf{v})\| \le L_1 \|\mathbf{u}-\mathbf{v}\|,\quad \forall \mathbf{u},\mathbf{v} \in \mathbb{R}^{n N},
\end{equation}
and there exists $\alpha>0$ such that 
\begin{equation}\label{cond1}
	\|(G'(\mathbf{x})^TG'(\mathbf{x})+\alpha H^TH)^{-1}G'^{T}(\mathbf{x})\|\le (L_1 c)^{-1}, \quad \forall \mathbf{x}\in \mathbb{R}^{n N},
	\end{equation}
	with $L_1$ being the Lipschitz constant from Equation (\ref{eqq8}) and $c>0$ being an upper bound on the norm of initial error $\|\mathbf{e}_0\|$.
\end{assumption}
We abuse the notation by denoting both vector and matrix norm by $\|\cdot\|$.
\begin{theorem}\label{th1}
	Let Assumption (\ref{ass1}) hold and let observations (Equation (\ref{eq2})) to be noise-free.
	Then norm of the error defined in Equation (\ref{err_def}) converges to zero as iteration goes to infinity.
\end{theorem}
\begin{proof}
For the sake of simplicity, we use $G_k$ and $G^{'}_k$ instead of $G(\mathbf{u}^{(k)})$ and $G'(\mathbf{u}^{(k)})$, respectively. Substituting Equation (\ref{eq6}) into Equation (\ref{err_def}) we have
	\begin{eqnarray*}
		\mathbf{e}_{k+1} &=& \mathbf{u}^{(k+1)} - \mathbf{u}^{\dagger}\\\nonumber
 &=& \mathbf{u}^{(k)} -(G_k^{'T}G^{'}_k+\alpha H^{T}H)^{-1}\left(G_k^{'T}G_k-\alpha H^{T}(\mathbf{y}-H\mathbf{u}^{(k)})\right) -\mathbf{u}^{\dagger}.            \end{eqnarray*}                                         By substituting Equation (\ref{eq2}) with $\eta_j=0, \ \forall  j=0,\cdots,N$ into the expression above, we obtain the following
	\begin{eqnarray*}\label{prop_th1}
		\mathbf{e}_{k+1}&=&\mathbf{e}_k - (G_k^{'T}G'_k+\alpha H^TH)^{-1}(G_k^{'T}G_k+\alpha H^TH \mathbf{e}_k).
			\end{eqnarray*}
Opening the brackets and using the following property 
\begin{eqnarray*}
I-(G_k^{'T}G'_k+\alpha H^T H)^{-1}\alpha H^TH = (G_k^{'T}G'_k+\alpha H^TH)^{-1}G_k^{'T}G'_k,
\end{eqnarray*}
we get                                                     
\begin{eqnarray*}\label{eq8}                 
\mathbf{e}_{k+1}=(G_k^{'T}G'_k+\alpha H^T H)^{-1}\left(G_k^{'T}(G'_k \mathbf{e}_k-G_k)\right).
\end{eqnarray*}
 Since  $G(\mathbf{u}^{\dagger})=0$,  we add it to the right-hand side and use the mean value theorem to obtain the following:
\begin{align*}
	\mathbf{e}_{k+1} &= (G_k^{'T}G'_k+\alpha H^T H)^{-1}(G_k^{'T} G'_k \mathbf{e}_k - G_k^{'T}(G(\mathbf{u}^{(k)})-G(\mathbf{u}^{\dagger}))\notag\\
	&=(G_k^{'T}G'_k+\alpha H^T H)^{-1} G_k^{'T} G'_k \mathbf{e}_k\notag\\
	&- (G_k^{'T}G'_k+\alpha H^T H)^{-1} G_k^{'T} \left( \int_0^1 G'(s\mathbf{u}^{(k)}+(1-s)\mathbf{u}^{\dagger}) ds \right)\mathbf{e}_k\notag\\
	&= (G_k^{'T}G'_k+\alpha H^T H)^{-1} G_k^{'T} \left( G'_k - \int_0^1 G'(s\mathbf{u}^{(k)}+(1-s)\mathbf{u}^{\dagger}) ds \right)\mathbf{e}_k\notag\\
	&=(G_k^{'T}G'_k+\alpha H^T H)^{-1} G_k^{'T} \left(\int_0^1 \left(G'(\mathbf{u}^{(k)}) - G'\left(s\mathbf{u}^{(k)}+(1-s)\mathbf{u}^{\dagger}\right) \right)\mathbf{e}_k ds\right).
\end{align*}
Taking norm of both sides and using Lipschitz continuity (Equation (\ref{eqq8})) on $G'$, we get
\[\|\mathbf{e}_{k+1} \| \le \frac{L_1}{2} \big\lVert (G_k^{'T}G'_k+\alpha H^{T}H)^{-1} G_k^{'T}\big\rVert \|\mathbf{e}_k\|^2.\]
Using  Equation (\ref{cond1}), we conclude that
\begin{eqnarray*}
\|\mathbf{e}_{k+1}\| \le \frac{1}{2 c} \|\mathbf{e}_k\|^2.
\end{eqnarray*}
Therefore, for $k=1,2,\dots$, we have 
\[\|\mathbf{e}_k\| \le \left(\frac{1}{2 c} \right)^{2^k -1} \|\mathbf{e}_0\|^{2^k}.\]
Since $\|\mathbf{e}_0\|\le c$, we get 
\begin{eqnarray*}\label{eq_star}
\|\mathbf{e}_k\| &\le&  \left(\frac{1}{2}\right)^{2^k-1} c,
\end{eqnarray*}
which leads to $\lim_{k\to\infty} \|\mathbf{e}_k\|=0$.
\end{proof}

\subsection{\textbf{Noisy observations}}
Next we consider noisy observations. We show that norm of the error defined in Equation (\ref{err_def}) is bounded. In order to prove this result, we require local conditions on $G'$ and bounded observation noise.
\begin{assumption}\label{ass_noisy1}
 $G$ is continuously differentiable in the open convex set $D\subset \mathbb{R}^{nN}$.
The Jacobian  $G'(\mathbf{u})$ defined in Equation (\ref{eq7}) is locally Lipschitz continuous in $D$, with the Lipschitz constant denoted by $L_2>0$, namely
\begin{equation}\label{Lip_noisy}
\|G'(\mathbf{u}) - G'(\mathbf{v})\| \le L_2 \|\mathbf{u}-\mathbf{v}\|, \quad \forall \mathbf{u},\mathbf{v} \in D,
\end{equation}
and there exists $0<\alpha<1$ such that 
\begin{align}\label{eq1_ass2}
  &\|(G'(\mathbf{x})^TG'(\mathbf{x})+\alpha H^TH)^{-1}G'^{T}(\mathbf{x})\|\le (L_2 c)^{-1}, \quad  \forall \mathbf{x}\in D,\\
   &\|H^T \eta\|\|(G'(\mathbf{x})^{T}G'(\mathbf{x})+\alpha H^TH)^{-1}\| \le c/2, \quad 	 \forall \mathbf{x}\in D,\label{eq2_ass2}
\end{align}
with $L_2$ being the Lipschitz constant from Equation (\ref{Lip_noisy}) and $c>0$ being an upper bound on the norm of initial error $\|\mathbf{e}_0\|$.
\end{assumption}
\begin{lemma}\label{lemma1}
(Dennis and Robert \cite{dennis1996numerical}) Let $G:\mathbb{R}^l\to\mathbb{R}^m$ be continuously differentiable in the open convex set $D\subset \mathbb{R}^l$ and let the Jacobian of $G$ be Lipschitz continuous at $\mathbf{x} \in D$, using a vector norm and the induced matrix operator norm and the Lipschitz constant $L$. Then, for any $\mathbf{x+p}\in D$,
\begin{equation*}\label{lem-eq}
    \|G(\mathbf{x+p})-G(\mathbf{x})-G'(\mathbf{x})\mathbf{p}\| \le \frac{L}{2}\|\mathbf{p}\|^2.
\end{equation*}
\end{lemma}
\begin{theorem}\label{th2}
Let Assumption (\ref{ass_noisy1}) hold and let observations (Equation (\ref{eq2})) to be noisy. 
Then 
\begin{equation}\label{eqq19}
\limsup_{k\to\infty}\| \mathbf{e}_k\| \le \frac{\alpha c}{1-\alpha},
\end{equation}
provided the iteration $\mathbf{u}^{(k)}$ does not leave the convex set $D$ for all $k$.
\end{theorem}
\begin{proof}
For the sake of simplicity, we use $G_k$ and $G^{'}_k$ instead of $G(\mathbf{u}^{(k)})$ and $G'(\mathbf{u}^{(k)})$, respectively. Substituting Equation (\ref{eq6}) into Equation (\ref{err_def}) we have
\begin{eqnarray*}\label{eq17}
e_{k+1} &=& \mathbf{u}^{(k+1)}-\mathbf{u}^{\dagger}\\\nonumber
&=&  \mathbf{u}^{(k)}-\mathbf{u}^{\dagger} - (G'^{T}_kG'_k +\alpha H^TH)^{-1} (G'^{T}_kG_k+\alpha H^T(H\mathbf{u}^{(k)}-\mathbf{y})) ,
\end{eqnarray*}
Using Equation (\ref{eq2}) in the above equation and adding and subtracting $G'^{T}_kG'_k \mathbf{e}_k$ leads to 
\begin{eqnarray*}\label{eq19}
\mathbf{e}_{k+1} &=& \mathbf{e}_k - (G'^{T}_kG'_k +\alpha H^TH)^{-1}(G'^{T}_kG_k+\alpha H^TH \mathbf{e}_k -\alpha H^T\eta))\\\nonumber
&=&  \mathbf{e}_k - (G'^{T}_kG'_k +\alpha H^TH)^{-1}(G'^{T}_k G'_k \mathbf{e}_k+\alpha H^TH \mathbf{e}_k+G'^{T}_k(G_k-G'_k \mathbf{e}_k) -\alpha H^T\eta))\\\nonumber
&=&  - (G'^{T}_k G'_k +\alpha H^TH)^{-1} (G'^{T}_k(G_k-G(\mathbf{u}^\dagger)-G'_k \mathbf{e}_k) -\alpha H^T\eta),
\end{eqnarray*}
since $G(\mathbf{u}^\dagger)=0$. By taking norm of both sides of the above equation and using Lemma \ref{lemma1}, we get
\begin{equation}\label{eq_eta}
    \|\mathbf{e}_{k+1}\| \le \frac{L_2}{2}\|(G'^{T}_k G'_k +\alpha H^TH)^{-1} G'^{T}_k\|\|\mathbf{e}_k\|^2 + \alpha \|(G'^{T}_k G'_k +\alpha H^TH)^{-1}\| \|H^T\eta\|.
\end{equation}
Using Equations (\ref{eq1_ass2}) and (\ref{eq2_ass2}) in the above expression leads to
\begin{equation*}
\|\mathbf{e}_{k+1}\| \le  \frac{ 1}{2 c} \|\mathbf{e}_k\|^2 + \alpha \frac{c}{2},
\end{equation*}
Using $(a+b)^2\le 2a^2 + 2b^2$ in the inequality above gives rise to
\begin{equation*}
\|\mathbf{e}_k\| \le    2^{-k} \left(\frac{1}{c}\right)^{2^k -1} \|\mathbf{e}_0\|^{2^k} + c \sum_{i=0}^{k-1} \left(\frac{\alpha^{2^i}}{2^{i+1}}\right),\quad\mbox{for}\quad k=1,2,\dots.
\end{equation*}
Since $\|\mathbf{e}_0\|\le c$ and $2^{-i} <1$ for $\forall i=0,1,\dots$, we get
\begin{equation*}\label{eq_new19}
\|\mathbf{e}_k\| <  2^{-k} c + c \sum_{i=0}^{k-1} \alpha^{2^i}.
\end{equation*}
Since $0<\alpha<1$, we have $\sum_{i=0}^{k-1}\alpha^{2^{i}} < \alpha \sum_{i=0}^{k}\alpha^{i}$ leading consequently to
\begin{equation*}\label{eq26}
\|\mathbf{e}_k\| <  2^{-k} c + c \alpha\sum_{i=0}^{k}\alpha^{i},
\end{equation*}
and in the limit of $k$ going to infinity
\begin{equation*}
 \limsup_{k\to\infty} \|\mathbf{e}_k\| < \frac{\alpha c}{1-\alpha}.
\end{equation*}
\end{proof}
\begin{remark}\label{remark}
Note that in case of $\eta_j=0, \forall j=1,\cdots,N$, we have $\|H^T \eta\|=0$
and Equation (\ref{eq2_ass2}) is trivially satisfied, which in turn implies assumption of $\alpha>0$ instead of $1>\alpha>0$. Furthermore, Equations (\ref{Lip_noisy}) and (\ref{eq1_ass2}) are equivalent to Equations (\ref{eqq8}) and (\ref{cond1}) but locally and thus Theorem (\ref{th2}) is
equivalent to Theorem (\ref{th1}) but locally.

\end{remark}

\section{JOINT STATE-PARAMETER ESTIMATION}\label{param}
Data assimilation can also be used if the dynamical model depends on uncertain parameter. We extend the Gauss-Newton DA method Equation (\ref{eq6}) to joint state-parameter estimation. We consider 
\begin{equation*}\label{new_G}
G(\mathbf{u};\bm{\theta})=u_{n+1}-F_n(u_n;{\theta}),
\end{equation*}
where $\bm{\theta}=(\theta_1,\theta_2,\dots,\theta_q)^{T}$ is an uncertain parameter. Then the minimization problem becomes
\begin{equation*}\label{eq19}
\min_{\{\mathbf{u}\in \mathbb{R}^{n N},\ \bm{\theta}\in \mathbb{R}^q\}} \{ \|G(\mathbf{u};\bm{\theta})\|^2 + \alpha \|H\mathbf{u}-\mathbf{y}\|^2\}.
\end{equation*}
Starting with an initial guess for the state $\mathbf{u}^{(0)}$ and parameters $\bm\theta^{(0)}$, the iteration proceeds as follows:
\begin{eqnarray}
\mathbf{u}^{(k+1)} &=&\mathbf{u}^{(k)} - \mathcal{L}(\mathbf{u}^{(k)};\bm\theta^{(k)}),\label{eq21}\\
\bm{\theta}^{(k+1)} &=&\bm\theta^{(k)} - \mathcal{S}(\mathbf{u}^{(k+1)};\bm\theta^{(k)}).\label{eq23}
\end{eqnarray}
Here
 $\mathcal{L}(\mathbf{u}^{(k)};\cdot)$ and $\mathcal{S}(\cdot; \bm{\theta}^{(k)})$
 are defined by
\begin{eqnarray*}
\mathcal{L}(\mathbf{u}^{(k)};\cdot) & = & \left(G_{u}^{'T}(\mathbf{u}^{(k)};\cdot) G_{u}^{'}(\mathbf{u}^{(k)};\cdot)+\alpha H^{T}H\right)^{-1} \bigg(G_{u}^{'T}(\mathbf{u}^{(k)};\cdot)G(\mathbf{u}^{(k)};\cdot)+ \alpha H^{T}(H\mathbf{u}^{(k)}-\mathbf{y})\bigg),\\
\mathcal{S}(\cdot; \bm{\theta}^{(k)}) & = &  \left(G_{\theta}^{'T}(\cdot;\bm\theta^{(k)})G_{\theta}^{'}(\cdot;\bm\theta^{(k)})\right)^{-1}
G_{\theta}^{'T}(\cdot;\bm\theta^{(k)}) G(\cdot;\bm\theta^{(k)}),
\end{eqnarray*}
respectively, $G'_u$ and $G'_{\theta}$ are derivatives of $G$ with respect to $u$ and $\theta$, respectively.

\begin{assumption}\label{ass_par}
We consider $G(\mathbf{u};\bm{\theta})= \mathcal{G}(\mathbf{u})+A\bm{\theta}$.
$\mathcal{G}$ is continuously differentiable in the open convex set $D\subset \mathbb{R}^{nN}$ and Lipshitz in $D$, with the Lipschitz constant denoted by $L_0>0$, namely
\begin{equation}\label{assth5_0}
    \|\mathcal{G}(\mathbf{u})-\mathcal{G}(\mathbf{v})\| \le L_0 \|\mathbf{u}-\mathbf{v}\|, \quad \forall \mathbf{u},\mathbf{v}\in D.
\end{equation}
The Jacobian  $\mathcal{G}'(\mathbf{u})$ is locally Lipschitz continuous in $D$, with the Lipschitz constant denoted by $L_3>0$, namely
\begin{equation}\label{Lip_par}
\|\mathcal{G}'(\mathbf{u}) - \mathcal{G}'(\mathbf{v})\| \le L_3 \|\mathbf{u}-\mathbf{v}\|, \quad \forall \mathbf{u},\mathbf{v} \in D,
\end{equation}
and there exists $\alpha>0$ such that 
\begin{equation}
  \|(\mathcal{G}'(\mathbf{x})^T\mathcal{G}'(\mathbf{x})+\alpha H^TH)^{-1}\mathcal{G}'^{T}(\mathbf{x})\|\le (2 L_3 c)^{-1}, \quad  \forall \mathbf{x}\in D,\label{assth5_1}
\end{equation}
with $L_3$ being the Lipschitz constant from Equation (\ref{Lip_par}),  and $c>0$ being an upper bound on the norm of initial error $\|\mathbf{e}_0\|$. Furthermore, $b/c<1$ where $b=\|A(A^TA)^{-1}A^{T}\|L_0/L_3$.
\end{assumption}

\begin{theorem}\label{th_par}
We consider $G(\mathbf{u};\bm{\theta})= \mathcal{G}(\mathbf{u})+A\bm{\theta}$ and noise-free observations of $\mathbf{u}$. We assume that Assumption \ref{ass_par} holds
for $G(\mathbf{u};\bm{\theta})$. Then
\begin{equation*}
     \limsup_{k\to\infty} \|\mathbf{e}_{k}\| < \frac{b/2}{1-b/c},
\end{equation*}
and
\begin{equation*}
    \limsup_{k\to\infty}\|\bm\theta^{(k)}-\bm\theta^{\dagger}\|<L_0\|(A^{T}A)^{-1} A^T\|  \frac{b/2}{1-b/c}.
\end{equation*}
\end{theorem}
Theorem (\ref{th_par}) is proven in (\ref{proof_th}).
\section{NUMERICAL EXPERIMENTS}\label{sec5}
In this section, we present numerical experiments to illustrate  the theoretical results of Section \ref{erran} and Section \ref{param}. 
First, we consider noise-free observations. We illustrate the convergence Theorem (\ref{th1}) under different sizes of observations. Second, we consider noisy observations and show the theoretical bound (Equation (\ref{eqq19})), computed numerically. 
Third, we consider a dynamical model with uncertain parameters and estimate them using the Gauss-Newton DA method (Equations (\ref{eq21}) and (\ref{eq23})). For each set-up, we perform 100 numerical experiments with different realizations of truth $\mathbf{u}^{\dagger}$, observations $\mathbf{y}$, and background solution $\mathbf{u}^b$. 
For each solution $\mathbf{u}$, we compute cost function, error with respect to the truth, error with respect to the truth of observed variables, and error with respect to the truth of non-observed variables :
 \begin{align}
C &= \|G(\mathbf{u})\| + \alpha \|\mathbf{y}-H\mathbf{u}\|,\label{eq43}\\ 
\mathcal{E} &= \|\mathbf{u}^{\dagger}-\mathbf{u}\|,\label{eq44}\\
\mathcal{E}^{O} &= \|H\mathbf{u}^{\dagger}-H\mathbf{u}\|,\label{eq45}\\
\mathcal{E}^{N} &= \|(I-H^T H)(\mathbf{u}^{\dagger}-\mathbf{u})\|,\label{eq46}
\end{align}
respectively.

We compare the Gauss-Newton DA method to WC4DVar, which minimizes the following cost function:
\begin{equation}\label{eq-wc}
J(u_0;\{y_n\}) = \frac{1}{2}\sum_{n=1}^N (y_n-Hu_n)^T R^{-1}(y_n-Hu_n) +\frac{1}{2} (u_n-F_n(u_{n-1})^T Q^{-1}(u_n-F_n(u_{n-1}),    
\end{equation}
where $R$ is covariance matrix of observational error and $Q$ is covariance matrix of model error (see, e.g., \cite{lewis1985use, talagrand1997assimilation, talagrand1987variational, sasaki1970some, courtier1994strategy}). 
The main distinction of the Gauss-Newton DA method from WC4DVar is in the tunable  parameter $\alpha$, which has a significant influence on error convergence or error boundedness. 
The minimization of the WC4DVar cost function is done by a Matlab built-in Levenberg-Marquardt algorithm.

We perform numerical experiments with the Lorenz 63 (L63) and Lorenz 96 (L96) models. L63 is a chaotic model which is widely used as a toy model in data-assimilation numerical experiments. It simulates atmospheric convection in a simple way \cite{Lorenz63}. The model is described by the following ODEs
\begin{equation}\label{eql63}
\dot{x}_1 = \sigma (x_2-x_1), \quad \dot{x}_2 = x_1 (\rho -x_3) - x_2,\quad \dot{x}_3 = x_1x_2 - b x_3.
\end{equation}
We implement the L63 model with the standard parameters, $\sigma=10$, $\rho=28$, and $b=8/3$. The differential equations are discretized with a forward Euler scheme with time step $\Delta t = 0.005$. The initial conditions are random numbers, which are independently and identically distributed.
We generate observations by computing a solution of L63 on $t\in[0,100]$. The observations are drawn at every tenth time step, which corresponds to assimilation time window of 6 hours.
For the Euler-discretized L63 model, the Lipschitz condition is
\[ \|G'(X)-G'(Y)\| \le \sqrt{2} \Delta t \|X-Y\|.\]
(For derivation of the Lipschitz constant see \ref{app1}). 

The L96 model \cite{lorenz1996predictability} is a  one-dimensional atmosphere model which is described by the following ODEs
\begin{equation}\label{L96}
\dot{x}_{l} = -x_{l-2}x_{l-1} + x_{l-1}x_{l+1} - x_l+\mathcal{F}, \quad l=1,\dots,d,
\end{equation}
where the dimension $d$ and forcing $\mathcal{F}$ are parameters. Cyclic boundary conditions are imposed. We implement the L96 model with the standard parameter choices $d=40$ and $\mathcal{F}=8$.
 We  discretize the differential equations with a forward Euler scheme with time step $\Delta t = 0.0025$. The initial conditions are random numbers which are independently and identically distributed. We generate observations by computing a solution of L96 on $t\in[0,100]$. The observations are drawn at every tenth time step, which corresponds to assimilation time window of 3 hours. The Lipschitz condition of the Euler-discretized L96 model is:
\[ \|G'(X)-G'(Y)\| \le \sqrt{6} \Delta t \|X-Y\|.\]
(For derivation of the Lipschitz constant see \ref{app2}).

\subsection{\textbf{State estimation given noise-free observations}}
We perform numerical experiments with noise-free observations, thus $\eta_j=0, \ j=0,\cdots,N$ in Equation (\ref{eq2}). In order to satisfy the conditions of Theorem (\ref{th1}), we first need to compute the Lipschitz constant of the Jacobian of the dynamical model. Next, we need to find a suitable $\alpha$. We use  Algorithm \ref{alg1}, where an upper bound on the initial error, $c$, is chosen arbitrarily.
\begin{algorithm}[H]
	\caption{Finding parameter $\alpha$ in case of noise-free observations} 
	Given the initial guess of the Gauss-Newton DA method, $\mathbf{u}^{(0)}$, an upper bound on the initial error, $c$, and the Lipschitz constant $L_1$ that satisfies Equation (\ref{eqq8}) in Assumption (\ref{ass1}), we choose an arbitrary positive $\alpha_0$, for example $\alpha_0=0.001$. 
	\begin{algorithmic}[1]
		\While {${\|(G'^{T}(\mathbf{u}^{(0)})G'(\mathbf{u}^{(0)})+\alpha_0 H^TH)^{-1}G'^{T}(\mathbf{u}^{(0)})\|> \left(L_1 c\right)^{-1}}$ and 
		
		$( \sim   isnan(\|(G'(\mathbf{u}^{(0)})^TG'(\mathbf{u}^{(0)})+\alpha_0 H^TH)^{-1}G'^{T}(\mathbf{u}^{(0)})\|))$}
		\State $\alpha_0\leftarrow 2 *\alpha_0$
		\EndWhile
				\If { $isnan(\|(G'(\mathbf{u}^{(0)})^T G'(\mathbf{u}^{(0)})+\alpha_0 H^TH)^{-1}G'^{T}(\mathbf{u}^{(0)})\|)$}
		        \State Error: There is no $\alpha$.
				\Else{}
				$\alpha \leftarrow \alpha_0$
			\EndIf
	\end{algorithmic} 
	\label{alg1}
\end{algorithm}

We use the same value of $\alpha$ throughout the iteration and check whether Equation (\ref{cond1}) is satisfied. If there exists an iteration $k$ such that Equation (\ref{cond1}) does not hold we terminate the iteration, otherwise the iteration proceeds until a tolerance value $(10^{-14})$ is reached for the change in the solution $\|\mathbf{u}^{(k+1)}-\mathbf{u}^{(k)}\|$.

We consider the L63 model and assume only the first component is observed, i.e., $H=[1,0,0]$. In all of the experiments Equation (\ref{cond1}) is satisfied throughout the iteration. In Figure \ref{fig1-a}, we display error with respect to the truth (Equation (\ref{eq44})) as a function of iteration. We plot median and plus and minus one standard deviation over 100 simulations. We see that error is a decreasing function, as we expect from Theorem (\ref{th1}).  We also considered other observation operators which depend only on second or third variable and obtained equivalent results (not shown). Therefore, for the L63 model the convergence Theorem (\ref{th1}) does not depend which variable is observed.
 
 \begin{figure}[htpb]
\centering
\begin{subfigure}{0.45\linewidth}
\includegraphics[width=\linewidth]{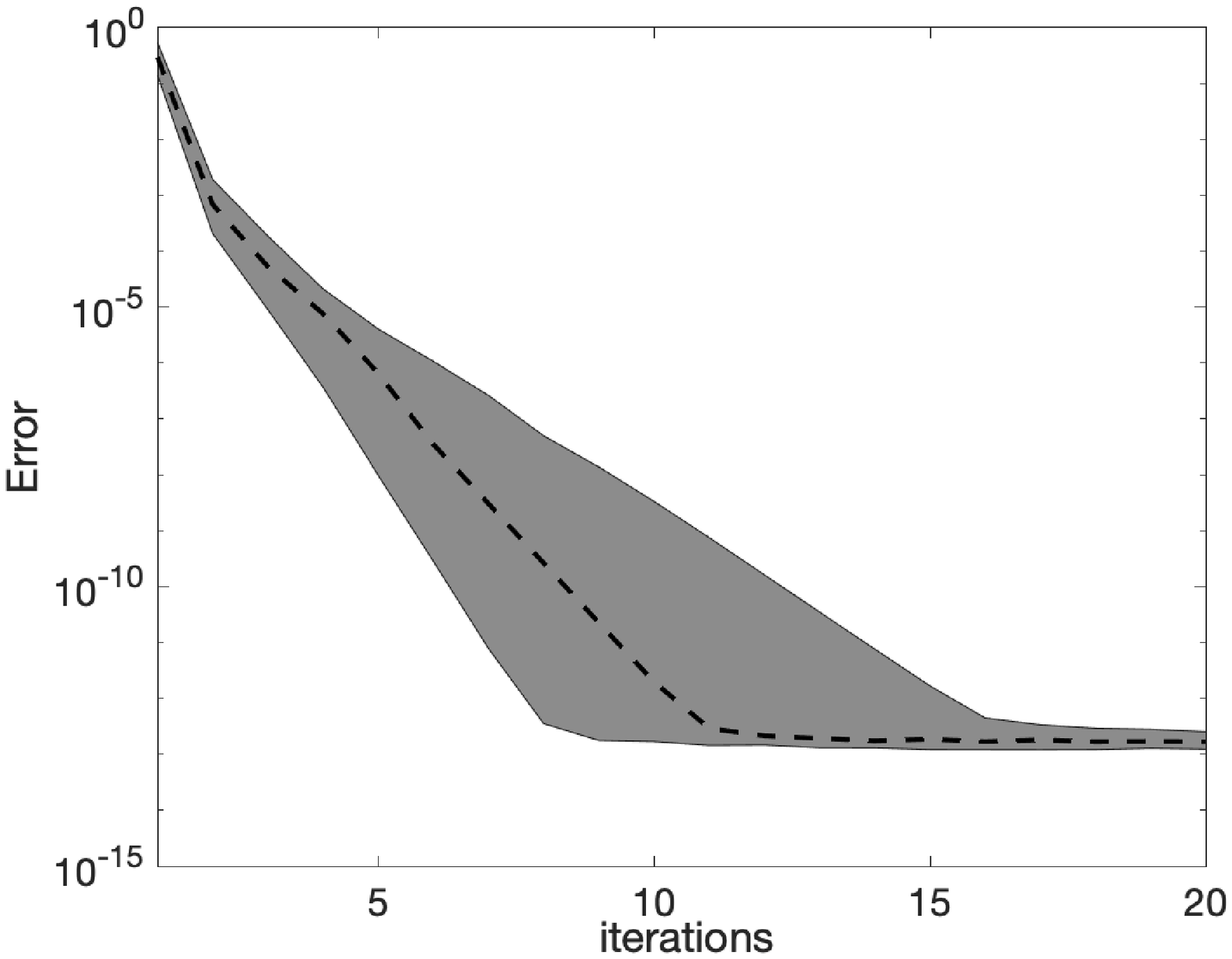}
\caption{}
\label{fig1-a}
\end{subfigure}
\begin{subfigure}{0.45\linewidth}
\includegraphics[width=\linewidth]{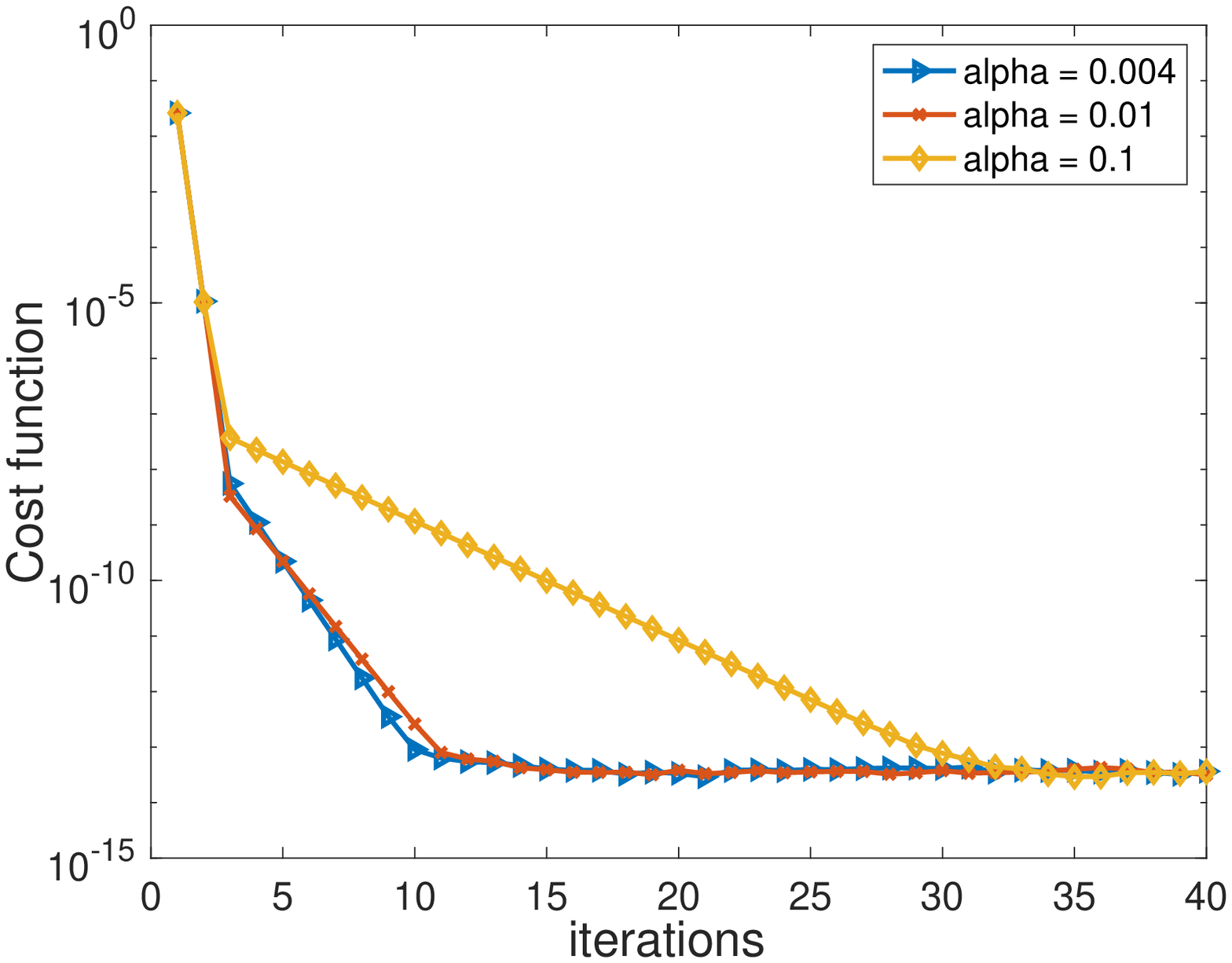}		
\caption{}
\label{fig1-b}
\end{subfigure}
\caption{Application of the Gauss-Newton DA method to L63 with noise-free observations, where only the first variable is observed. On the left: Error (Equation (\ref{eq44})) as a function of iteration: median (dashed line), +/- one standard deviation (shadowed area) over 100 simulations. On the right: Cost function (Equation (\ref{eq43})) as a function of iteration for different values of $\alpha$. }
\label{fig1}
\end{figure}
Furthermore, We perform  numerical experiments with different values of $\alpha$ that satisfy Equation (\ref{cond1}).
In Figure \ref{fig1-b}, we display cost function (Equation (\ref{eq43})) as a function of iteration for different values of $\alpha$.  A faster convergence rate is achieved at $\alpha=0.004$, as to be expected since $\alpha=0$ corresponds to the second order Newton method when observations are complete.

Next, we perform numerical experiments with the L96 model and we assume every second variable is observed.
In Figure \ref{fig2-a}, we display error with respect to the truth as a function of iteration. We plot median and plus and minus one standard deviation over 100 simulations. We see that error is a decreasing function, as we expect from Theorem (\ref{th1}). Comparing to Figure \ref{fig1-a} we observe that more iterations are needed to reach the desired tolerance in the L96 model due to the higher dimension.
 \begin{figure}[htpb]
\centering
\begin{subfigure}{0.45\linewidth}
\includegraphics[width=\linewidth]{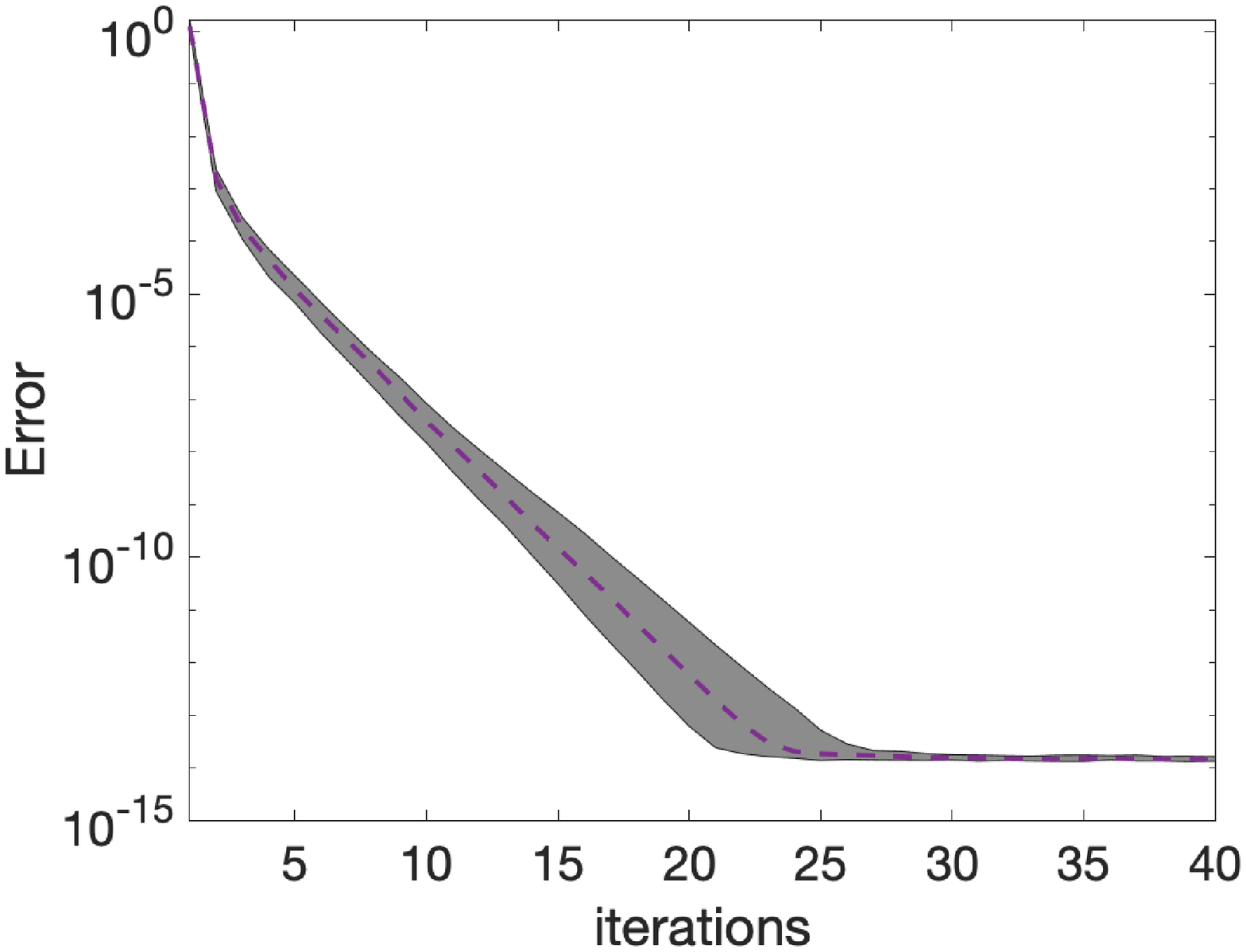}
\caption{}
\label{fig2-a}
\end{subfigure}
\begin{subfigure}{0.45\linewidth}
 \includegraphics[width=\textwidth]{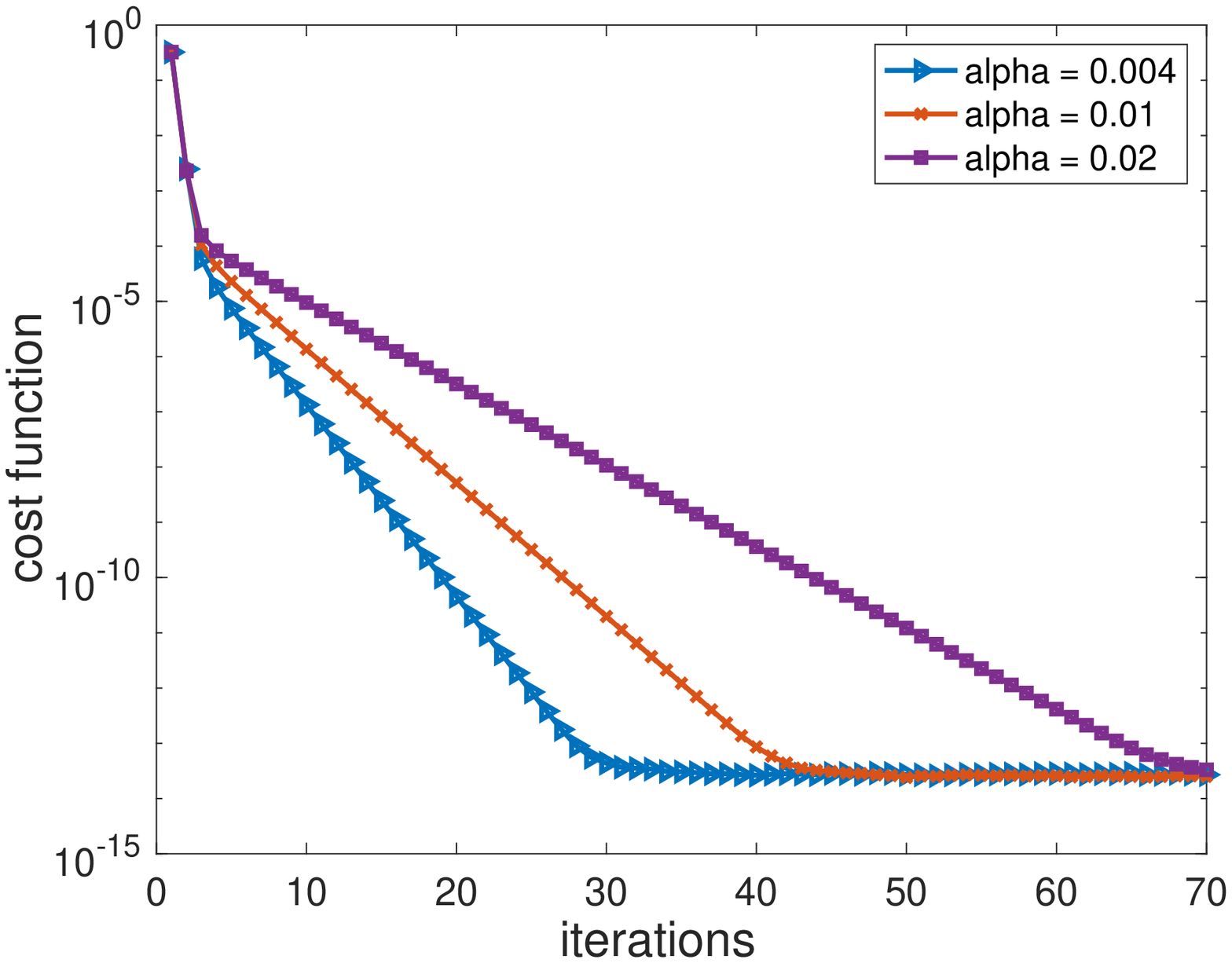}  
\caption{}
\label{fig2-b}
\end{subfigure}
\caption{Application of the Gauss-Newton DA method to L96 with noise-free observations, where every second variable is observed. On the left: Error (Equation (\ref{eq44})) as a function of iteration: median (dashed line), +/- one standard deviation (shadowed area) over 100 simulations. On the right: Cost function (Equation (\ref{eq43})) as a function of iteration for different values of $\alpha$.}
\label{fig2}
\end{figure}
In Figure \ref{fig2-b}, we plot the cost function of the L96 model for different values of $\alpha$ when  observing every second variable. A faster convergence rate is achieved at $\alpha=0.004$.
Furthermore, we investigate the convergence property under different sizes of observations. In Figure \ref{fig3}, we plot the median value of error (Equation (\ref{eq44})) over 20 simulations as a function of iteration in a semi-log scale. We see that when the size of observations decreases more iterations are needed to obtain the desired tolerance. 
 \begin{figure}[htpb]
\centering
\begin{subfigure}{0.45\linewidth}
\includegraphics[width=\linewidth]{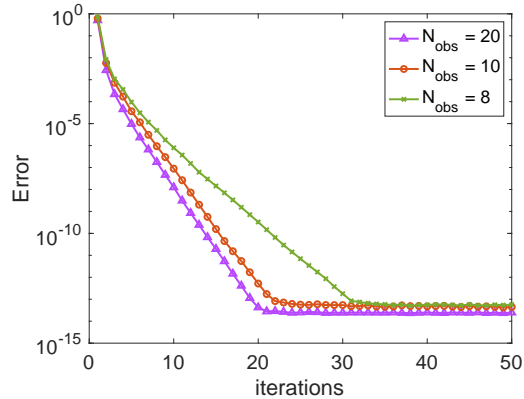}
\end{subfigure}
\caption{Application to L96 with noise-free observations. Error of the Gauss-Newton DA method as a function of iterations for different sizes of observations.}
\label{fig3}
\end{figure}

Furthermore, Equation (\ref{cond1}) of Assumption (\ref{ass1}) gives implicit requirement on the size of observations. In numerical experiments with the L96 model, we see that if the number of observations is less than eight then the matrix $(G'(\mathbf{x})^T G'(\mathbf{x})+\alpha H^TH)^{-1}$ is ill-posed and we are not able to find $\alpha$ that satisfies Equation (\ref{cond1}). Therefore, the conditions in Theorem (\ref{th1}) do not hold and error convergence is not guaranteed.

\subsection{\textbf{State estimation given noisy observations}}
We perform numerical experiments with noisy observations, where $\eta \sim \mathcal{N}(0,\gamma^2 I)$ in Equation (\ref{eq2}). 
We assume that $\|H^T\eta\|$ is bounded and known, which is a strong assumption.
In order to satisfy the conditions of Theorem (\ref{th2}), we first need to compute the Lipschitz constant of the Jacobian of the dynamical model, and then to find a suitable $\alpha$. We use  Algorithm \ref{alg2} to find $\alpha$, where an upper bound on the initial error, $c$, is chosen arbitrarily.
\begin{algorithm}[H]
	\caption{Finding parameter $\alpha$ in case of noisy observations} 
Given the initial guess of the Gauss-Newton DA method, $\mathbf{u}^{(0)}$, an upper bound on the initial error, $c$, and the Lipschitz constant $L_2$ that satisfies  Equation (\ref{Lip_noisy}) in Assumption (\ref{ass_noisy1}), we choose an arbitrary positive $\alpha_0$, for example $\alpha_0=0.001$. 
	\begin{algorithmic}[1]
		\While {${\|(G'^{T}(\mathbf{u}^{(0)})G'(\mathbf{u}^{(0)})+\alpha_0 H^TH)^{-1}G'^{T}(\mathbf{u}^{(0)})\|> (L_2 c)^{-1}}$ and 
		
		${\|H^T \eta\|\|(G'^{T}(\mathbf{u}^{(0)})G'(\mathbf{u}^{(0)})+\alpha_0 H^TH)^{-1}\|> c/2}$
		and
		
		$( \sim   isnan(\|(G'(\mathbf{u}^{(0)})^TG'(\mathbf{u}^{(0)})+\alpha_0 H^TH)^{-1}G'^{T}(\mathbf{u}^{(0)})\|))$}
		\State $\alpha_0\leftarrow 2 *\alpha_0$
		\EndWhile
				\If { $isnan(\|(G'(\mathbf{u}^{(0)})^T G'(\mathbf{u}^{(0)})+\alpha_0 H^TH)^{-1}G'^{T}(\mathbf{u}^{(0)})\|)$
				or
				$\alpha_0>1$}
		        \State Error: There is no $\alpha$.
				\Else{}
				$\alpha \leftarrow \alpha_0$
			\EndIf
	\end{algorithmic} 
	\label{alg2}
\end{algorithm}

We use the same value of $\alpha$ throughout the iteration and check whether Equations (\ref{eq1_ass2}) and (\ref{eq2_ass2}) are satisfied. If there exists an iteration $k$ such that either Equation (\ref{eq1_ass2}) or Equation (\ref{eq2_ass2}) does not hold, we terminate the iteration, otherwise the iteration proceeds until it reaches the maximum value of iterations, which is 20 for L63 and 70 for L96.

We consider the L63 model with noisy observations. We assume only the first component is observed, i.e., $H=[1,0,0]$ and $\gamma=0.01$ in Equation (\ref{eq2}). In all of the experiments Equations (\ref{eq1_ass2}) and (\ref{eq2_ass2}) are satisfied throughout the iteration. In Figure \ref{fig4}, we display error with respect to the truth (Equation (\ref{eq44})) as a function of iteration. We plot median (dashed line) and plus and minus one standard deviation (shadowed area) over 100 simulations. In Figure \ref{fig4}, we also display theoretical bound (Equation (\ref{eqq19})), we plot median (green line) over 100 simulations which depends on different $\alpha$s.
We see that error is below the theoretical bound, as we expect from Theorem (\ref{th2}). 
On the right of Figure \ref{fig4}, we display error of observed components (Equation (\ref{eq45})) as a function of iteration. We plot median and plus and minus one standard deviation over 100 simulations. We see that error of observed components (dashed line) is less than observation error $\|\mathbf{y}-H\mathbf{u}^{\dagger}\|$ (red line).  
Furthermore, comparing the right panel to the left panel of Figure \ref{fig4} we see that the error of observed components is less than the total error.

Next, we perform numerical experiments with the L96 model with noisy observations.  We assume every second variable is observed and $\gamma=0.01$ in Equation (\ref{eq2}). On the left of Figure \ref{fig5}, we display error with respect to the truth as a function of iteration. We plot median and plus and minus one standard deviation over 100 simulations. We see that error is a decreasing function for the first five iterations and after that it remains bounded. Comparing to Figure \ref{fig4} we observe that due to the higher dimensions, more iterations are needed to reach the desired tolerance in the L96 model. Moreover, as we expect from Theorem (\ref{th2}), the median of error is below the median of theoretical bound (green line) over 100 simulations. On the right of Figure \ref{fig5}, we display error of observed components (Equation (\ref{eq45})) as a function of iteration. We see that error of observed components is less than the total error (left panel of Figure \ref{fig5}),  as it is expected. Furthermore, we see from both Figures \ref{fig4} and \ref{fig5} that the theoretical bound is not tight as to be expected. 

\begin{figure}[htpb]
\centering
		\begin{subfigure}{1\linewidth}
			\includegraphics[width=\linewidth]{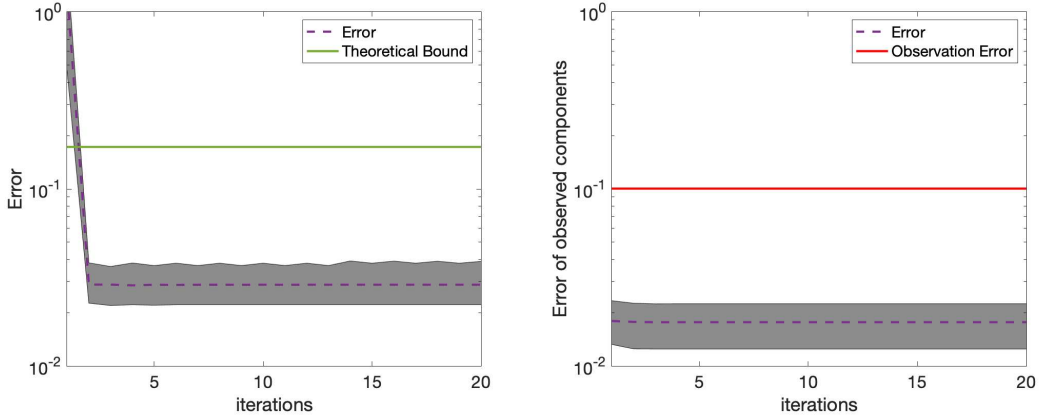}
		\end{subfigure}
		\caption{Application of the Gauss-Newton DA method to L63 with noisy observations, where only the first variable is observed and $\gamma=0.01$ in Equation (\ref{eq2}). On the left: Error (Equation (\ref{eq44})) as a function of iteration: median (dashed line), +/- one standard deviation (shadowed area) over 100 simulations and a theoretical bound (Equation (\ref{eqq19})) in green. On the right: Error of observed components (Equation (\ref{eq45})) as a function of iteration and observation error in red.}
				\label{fig4}
	\end{figure}

	\begin{figure}[htpb]
\centering
		\begin{subfigure}{1\linewidth}
			\includegraphics[width=\linewidth]{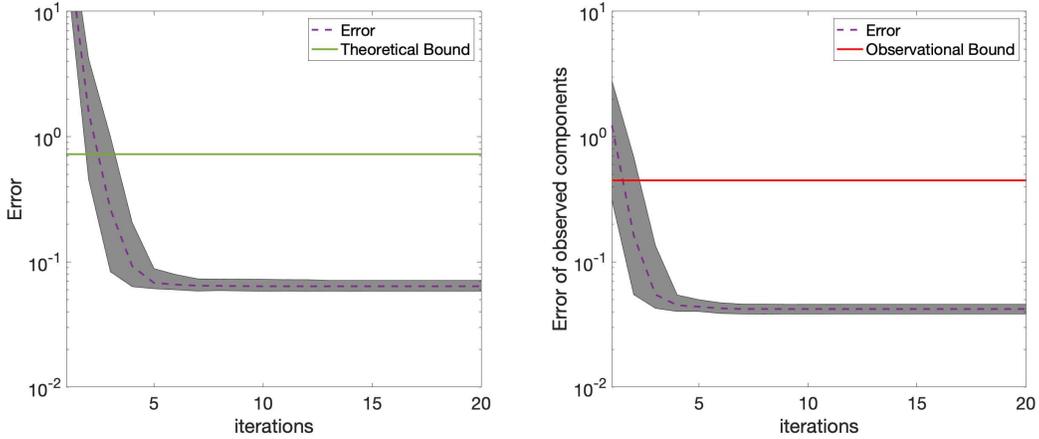}
		\end{subfigure}
		\caption{Application of the Gauss-Newton DA method to L96 with noisy observations, where every second variable is observed and $\gamma=0.01$ in Equation (\ref{eq2}). On the left: Error (Equation (\ref{eq44})) as a function of iteration: median (dashed line), +/- one standard deviation (shadowed area) over 100 simulations and a theoretical bound Equation (\ref{eqq19}) in green. On the right: Error of observed components (Equation (\ref{eq45})) and observation error in red.}
      \label{fig5}
	\end{figure}
	
	Finally, we investigate error (Equation (\ref{eq44})) as a function of observation noise level $\gamma$.
     From Theorem (\ref{th2}) follows that the convergence result is achieved as $\|H^{T}\eta\|$ goes to zero. 
     In Figure \ref{fig6}, we display error (Equation (\ref{eq44})) for different values of $\gamma$ for the L63 model on the left and for the L96 model on the right. As we expected from Remark (\ref{remark}), we see that error decreases as observation noise level decreases and the convergence result achieves when $\gamma$ is close to zero.

\begin{figure}[htpb]
\centering
\begin{subfigure}{1\linewidth}
\includegraphics[width=\linewidth]{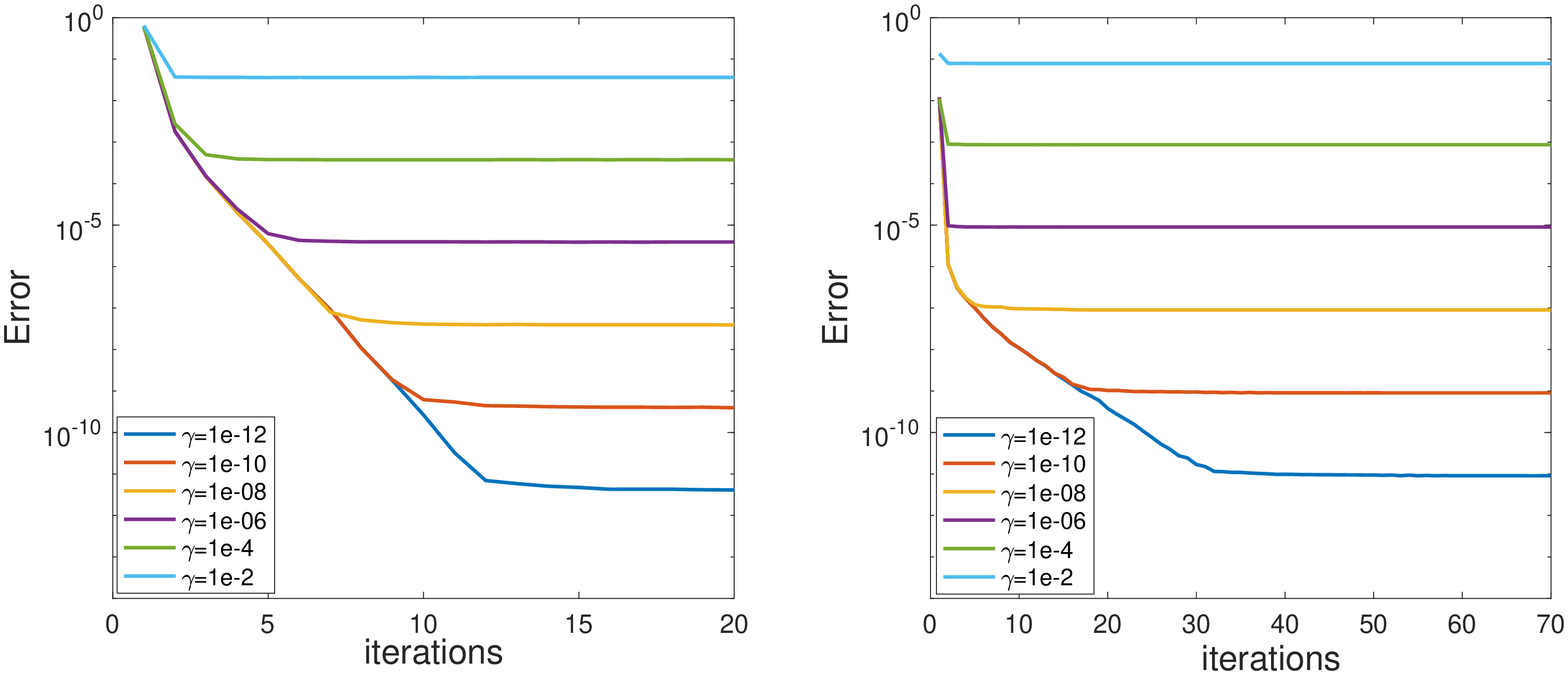}
\end{subfigure}
\caption{Error with respect to the truth for different observation noise level. On the left: Application to L63 with noisy observations, we observe only the first variable. On the right: Application to L96 with noisy observations, we observe every second variable.}
\label{fig6}
\end{figure}

\subsection{\textbf{Comparison to weak-constraint 4DVar}}
 Weak-constraint four-dimensional variational method is one of the well-known data-assimilation methods for estimating initial condition in weather forecasting applications. It minimizes the cost function (Equation (\ref{eq-wc})) under assumption of imperfect model dynamics which is also the goal of the Gauss-Newton DA method. We compare 
 Gauss-Newton DA method to WC4DVar method.
 We perform numerical experiments using the L63 and L96 models with the same parameters as in the previous section. In these experiments,  we use identical data, models, and windows for both methods.

In Figure \ref{fig7} and Figure \ref{fig8} we plot errors for the L63 and L96 models, respectively. On the left of the figures, we plot errors with respect to the truth of observed variables (Equation (\ref{eq45})). On the right of the figures, we plot errors with respect to the truth of non-observed variables (Equation (\ref{eq46})). We see that error of Gauss-Newton DA method is significantly less than the error of WC4DVar method for both observed variables and non-observed variables. We see that the error of both methods is below the observation error.

\begin{figure}[htpb]
    \centering
    \begin{subfigure}{1\textwidth}
    \includegraphics[width=\textwidth]{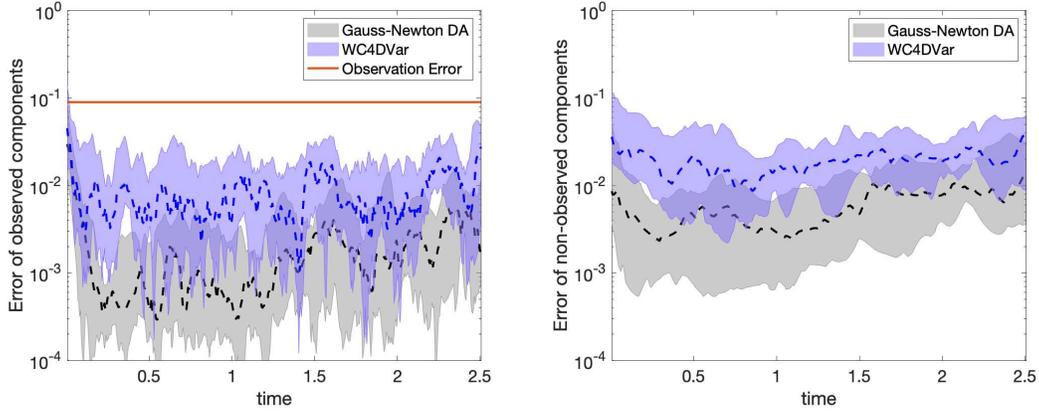}
    \label{fig:my_label}
    \end{subfigure}
     \caption{Application to L63. Error as a function of time: median (dashed line) +/- one standard deviation over 20 simulations with length of assimilation window, 2.5. On the left: error with respect to the truth of observed variables. On the right: error with respect to the truth of non-observed variables. The Gauss-Newton DA method is in grey, WC4DVar method is in blue, and the observational error is in red.}
     \label{fig7}
\end{figure}
\begin{figure}[htpb]
    \centering
    \begin{subfigure}{1\textwidth}
    \includegraphics[width=\textwidth]{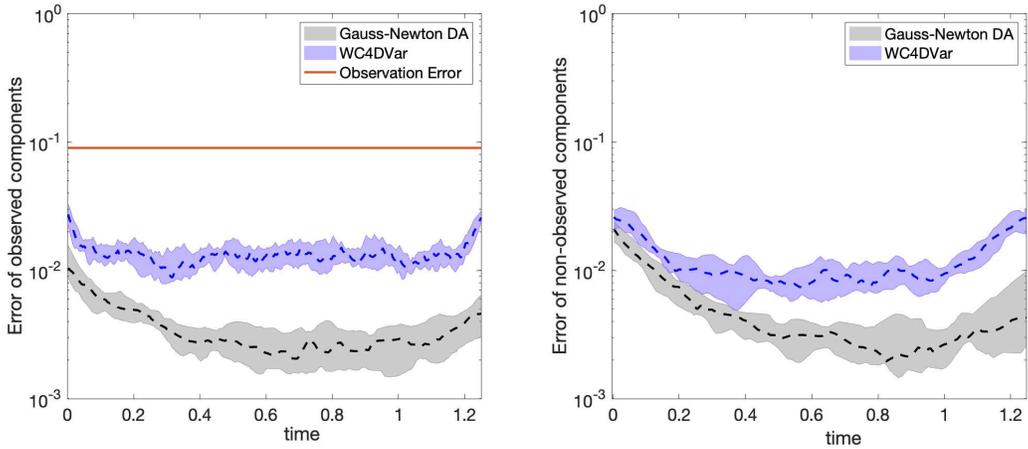}
    \end{subfigure}
     \caption{Application to L96. Error as a function of time: median (dashed line) +/- one standard deviation over 20 simulations with length of assimilation window, 1.25. On the left: error with respect to the truth of observed variables. On the right: error with respect to the truth of non-observed variables. The Gauss-Newton DA method is in grey, WC4DVar method is in blue, and the observational error is in red.}
     \label{fig8}
\end{figure}

\subsection{\textbf{Parameter estimation}}
As described in Section \ref{param}, the Gauss-Newton DA method can be applied to the problem of joint state-parameter estimation. In these experiments, we use the Gauss-Newton DA method (Equations (\ref{eq21}) and (\ref{eq23})) to estimate a parameter $\sigma$ of the L63 model which we assume to be uncertain. Different values of initial $\sigma$ were chosen, 5, 15, and 20, with the true $\sigma$ being 10. We alternate between Equations (\ref{eq21}) and (\ref{eq23}) until a termination condition is satisfied. The termination condition is satisfied when either the number of iterations reaches the maximum (500) or the distance between two successive approximations of the uncertain parameter is less than desired tolerance $(10^{-3})$. We consider both noisy and noise-free observations and assume the first and second components of the state are observed. For each numerical experiment and fixed initial $\sigma$, we perform 100 realizations. In Table \ref{Tab:1} and  \ref{Tab:2}, we display the error of state estimation, which is the median of Equation (\ref{eq44}) over 100 simulations, and estimated $\sigma$ for different initial parameters. Table  \ref{Tab:1} is for noise-free observations and Table  \ref{Tab:2} is for noisy observations with $\gamma=\sqrt{3}$.

In Table \ref{Tab:1}, we consider noise-free observations. We assume the first and second components of the state are observed. The results of $\sigma$ estimation are shown in Table \ref{Tab:1}. 

In Table \ref{Tab:2}, we consider noisy observations.
The observations of state are obtained from the true state by adding i.i.d zero mean Gaussian noise as in Equation (\ref{eq2}) with covariance matrix $\Gamma=\gamma^2 I$, where $\gamma^2$ is the variance of the noise procedure. In this numerical experiment, we consider noisy observations with $\gamma=\sqrt{3}$ in Equation (\ref{eq2}). 
 It should be noted that similar results can be obtained for $\rho$  or $\beta$  estimation of the L63 model (not shown).
\begin{table}[htbp]
\caption{Estimation of $\sigma$ by the Gauss-Newton DA method for noise-free observations, we observe  the first and second components of the state.}
\label{Tab:1}
\begin{center}
\begin{tabular}{|c|c|c|c|}
\hline
Property & \multicolumn{3}{p{5cm}|}{\centering Initial guess for $\sigma$} \\
\hline
Initial $\sigma$ & 5 & 15 & 20  \\ \hline
Error between estimated and true solution & $0.4025$ & $0.4256$ & $0.4267$ \\
Estimated $\sigma$ & 9.7465 &  10.2746 & 10.2785 \\
\hline
\end{tabular}
\end{center}
\end{table}
\begin{table}[htbp]
\caption{Estimation of $\sigma$ by Gauss-Newton DA method for noisy observations.}
\label{Tab:2}
\begin{center}
\begin{tabular}{|c|c|c|c|}
\hline
Property & \multicolumn{3}{p{5cm}|}{\centering Initial guess for $\sigma$} \\
\hline
Initial $\sigma$ & 5 & 15 & 20 \\ \hline
Observation error & \multicolumn{3}{c|}{3}\\ \hline 
Error between estimated and true solution & $2.7287$ & $6.7806$ & $3.8887$ \\
Estimated $\sigma$ & 9.7380 &  10.4505 & 10.8239 \\
\hline
\end{tabular}
\end{center}
\end{table}

\section{CONCLUSION}\label{conc}
We considered a cost function of a data-assimilation problem
consisting of observation mismatch, $\|H\mathbf{u}-\mathbf{y}\|$ weighted  by a parameter $\alpha$ and model mismatch $\|G(\mathbf{u})\|$. By solving this minimization problem approximately using an iterative algorithm, Gauss-Newton DA method, we found an estimate for a considered system. One of the advantage of this cost function is using the information of available observations in each iteration which can help to find a more accurate approximated solution for a physical system. 
We then established error bounds with respect to the truth under some conditions in two cases, noise-free and noisy observations. In case of noise-free observations, we proved the method is convergent and we obtained an upper bound in case of noisy observations.
In numerical experiments, we applied Gauss-Newton DA method for two models, Lorenz 63 and Lorenz 96. We observed the numerical evidence confirmed the theoretical results and we demonstrated that the results compare favorably to those of WC4DVar.
Moreover, we showed that how the size of observations could affect on the convergence results. We observed that if the size of observations is less than eight in the Lorenz 96 model, we do not have a convergence result. Therefore future directions include finding a method for deriving an approximation based on low-order of observations.
\begin{appendix}
\section{Lipschitz constants for the Lorenz 63 and the Lorenz 96 models}
\addcontentsline{toc}{section}{Appendices}
\renewcommand{\thesubsection}{A.\arabic{subsection}}
Here we compute the Lipschitz constant for the Lorenz 63 and the Lorenz 96 models discretized with forward Euler. 
\subsection{Lipschitz constant for the L63 model discretized with forward Euler}\label{app1}
Consider the discretized form of Equation \eqref{eql63} 
\begin{eqnarray}\label{disc_l63}
X_1(t_{k+1}) &=& X_1(t_k) + \Delta t \sigma (X_2(t_k)-X_1(t_k)),\\ \nonumber
X_2(t_{k+1}) &=& X_2(t_k) + \Delta t\left( X_1(t_k) (\rho -X_3(t_k)) - X_2(t_k)\right),\\ \nonumber
X_3(t_{k+1}) &=& X_3(t_k) +\Delta t \left(X_1(t_k) X_2(t_k) - bX_3(t_k)\right),\nonumber
\end{eqnarray}
with the following matrix notation
\[
\mathbf{X}_{k+1} =F(\mathbf{X}_k), \quad G(\mathbf{X}):= \mathbf{X}_{k+1} -F(\mathbf{X}_k)
\]
where $\mathbf{X}_k = \left(X_1(t_k),X_2(t_k),X_3(t_k) \right)^T$
and $F(\mathbf{X}_k)$ is the right-hand side of \eqref{disc_l63}. From the definition of Jacobian of $G$ we have
\begin{equation*}
G'(\mathbf{X}) -G'(\mathbf{Y}) = 
\begin{bmatrix}
-F'(\mathbf{X}_0)+F'(\mathbf{Y}_0)& &\\
 & -F'(\mathbf{X}_1)+F'(\mathbf{Y}_1) &  &\\
  & & \ddots &  & \\
& & & -F'(\mathbf{X}_{N-1})+F'(\mathbf{Y}_{N-1}) & \\
\end{bmatrix}.
\end{equation*}
with
\begin{equation}\label{eq36}
F'(\mathbf{X}_k) = \begin{pmatrix}
1-\sigma \Delta t & \sigma\Delta t& 0\\
\Delta t (\rho-X_3(t_k))& 1-\Delta t & - \Delta t X_1(t_k) \\
\Delta t X_2(t_k) & \Delta t X_1(t_k)& 1-b\Delta t\\
\end{pmatrix}.
\end{equation}
From (\ref{eq36}) we get
\begin{align}
\|F'(\mathbf{X}_k)-F'(\mathbf{Y}_k)\|^2_F = \sum_{j=1}^n\sum_{i=1}^m |F_{ij}|^2\ &=  \Delta t^2 \bigg( |Y_3(t_k)-X_3(t_k)|^2 +  |Y_1(t_k)-X_1(t_k)|^2 \notag\\
&+   |Y_2(t_k)-X_2(t_k)|^2 +  |Y_1(t_k)-X_1(t_k)|^2\bigg)\notag\\
& \le  2\Delta t^2 \|\mathbf{X}_k - \mathbf{Y}_k\|^2_{L^2}.
\end{align}
Therefore, we have
\begin{equation*}
\|G'(\mathbf{X}) -G'(\mathbf{Y})\|_F \le \sqrt{2} \Delta t \|\mathbf{X}-\mathbf{Y}\|_{L^2}.
\end{equation*}
Since $\|A\|_{L^2} \le \|A\|_F$, by substituting it into the expression above we obtain the following
\begin{equation*}
\|G'(\mathbf{X}) -G'(\mathbf{Y})\|_{L^2} \le \sqrt{2} \Delta t \|\mathbf{X}-\mathbf{Y}\|_{L^2}.
\end{equation*}

\subsection{Lipschitz constant for the L96 model discretized with forward Euler}\label{app2}
Consider the discretized form of Equation \eqref{L96} 
\begin{equation}\label{disc_l96}
  X_l (t_{k+1}) = X_l(t_k) + \left(-X_{l-2}(t_k) X_{l-1}(t_k) + X_{l-1}(t_k)X_{l+1}(t_k)-X_l(t_k)+\mathcal{F}\right)\Delta t,\ l=1,\dots,40.  
\end{equation}
with the following matrix notation
\[\mathbf{X}_{k+1} =F(\mathbf{X}_k), \qquad G(\mathbf{X}):= \mathbf{X}_{k+1} -F(\mathbf{X}_k)\]
where $\mathbf{X}_{k} = \left(X_1(t_{k}),\dots,X_{40}(t_{k}) \right)^T$
 and $F(\mathbf{X}_k)$ is the right-hand side of \eqref{disc_l96}.
From the definition of Jacobian of $G$ we have
\begin{equation*}
G'(\mathbf{X}) -G'(\mathbf{Y}) = 
\begin{bmatrix}
-F'(\mathbf{X}_0)+F'(\mathbf{Y}_0)& &\\
 & -F'(\mathbf{X}_1)+F'(\mathbf{Y}_1) &  &\\
  & & \ddots &  & \\
& & & -F'(\mathbf{X}_{N-1})+F'(\mathbf{Y}_{N-1}) & \\
\end{bmatrix}.
\end{equation*}
with
 \begin{equation*}
F'(\mathbf{X}_k) = \begin{pmatrix}
1-\Delta t & \Delta t X_{40} & 0& \dots & 0 & -\Delta t X_{40} & \Delta t (X_2-X_{39})\\
\Delta t (X_3-X_{40}) & 1-\Delta t & \Delta t X_1 & 0 &\dots &0 &-\Delta t X_1 \\
\vdots &\dots & &\dots & & &\vdots \\
0 &\dots&0&-\Delta X_{38}&\Delta t (X_{40}-X_{37})& 1-\Delta t& \Delta t X_{38}\\
\Delta t X_{39} & 0 &\dots &0&-\Delta t X_{39} & \Delta t (X_1-X_{38}) & 1-\Delta t\\
\end{pmatrix}.
\end{equation*}
From the expression above we get
\begin{align*}
\|F'(\mathbf{X}_k)-F'(\mathbf{Y}_k)\|_F^2 &= \Delta t^2 |X_{40}-Y_{40}|^2 + \Delta t^2 |X_{40}-Y_{40}|^2 \notag\\
&+ \Delta t^2 |(X_2-Y_2)-(X_{39}-Y_{39})|^2 \notag\\
&+ \Delta t^2 |(X_3-Y_3)-(X_{40}-Y_{40})|^2  \notag\\
&+ \Delta t^2 |X_1-Y_1|^2 + \Delta t^2 |X_1-Y_1|^2 \notag\\
&+ ... \notag\\
&+ \Delta t^2 |X_{38}-Y_{38}|^2 + \Delta t^2 |(X_{40}-Y_{40})-(X_{37}-Y_{37})|^2  \notag\\
&+\Delta t^2 |X_{38}-Y_{38}|^2 + \Delta t^2 |X_{39}-Y_{39}|^2\notag\\
&+ \Delta t^2 |X_{39}-Y_{39}|^2 +\Delta t^2 |(X_{1}-Y_{1})-(X_{38}-Y_{38})|^2\notag\\
&\le 2\Delta t^2 \|\mathbf{X}_k - \mathbf{Y}_k\|^2_{L^2} \notag\\
&+ \{2\Delta t^2 |X_2-Y_2|^2 +2\Delta t^2 |X_{39}-Y_{39}|^2 +2\Delta t^2 |X_3-Y_3|^2\notag\\
 &+ 2\Delta t^2 |X_{40}-Y_{40}|^2 +...+2\Delta t^2 |X_{40}-Y_{40}|^2 \notag\\
 &+ 2\Delta t^2 |X_{37}-Y_{37}|^2 + 2\Delta t^2 |X_{1}-Y_{1}|^2 + 2\Delta t^2 |X_{38}-Y_{38}|^2\}\notag\\
 &\le 2\Delta t^2 \|\mathbf{X}_k - \mathbf{Y}_k\|^2_{L^2} + 2\Delta t^2 (2\|\mathbf{X}_k - \mathbf{Y}_k\|^2_{L^2})\notag\\
 &= 6 \Delta t^2 \|\mathbf{X}_k - \mathbf{Y}_k\|^2_{L^2}.
\end{align*}
On the other hand,
\begin{equation*}
\|G'(\mathbf{X}_k) -G'(\mathbf{Y}_k)\|^2_{L^2} \le\|G'(\mathbf{X}_k) -G'(\mathbf{Y}_k)\|^2_F = \sum_{j=1}^n\sum_{i=1}^m |F_{ij}|^2. 
\end{equation*}
Combining the two above expressions leads to
\[ 
\|G'(\mathbf{X}) -G'(\mathbf{Y})\|_{L^2} \le \sqrt{6} \Delta t \|\mathbf{X} - \mathbf{Y}\|_{L^2}.
\]
\section{Bound on the parameter estimation error}\label{proof_th}
Here we prove Theorem (\ref{th_par}).

\begin{proof}
By setting $G(\mathbf{u}^{(k)};\bm{\theta}^{(k-1)})= \mathcal{G}(\mathbf{u}^{(k)})+A\bm\theta^{(k-1)}$ in Equation (\ref{eq23}), we get
\begin{equation*}
    \bm\theta^{(k)} = \bm\theta^{(k-1)} - (A^{T}A)^{-1} A^{T}\left( \mathcal{G}(\mathbf{u}^{(k)})+A \bm\theta^{(k-1)}\right)
\end{equation*}
Let us define $\mathbf{E}_{k} = \bm\theta^{(k)}-\bm\theta^\dagger$, then we have
\begin{align*}
    \mathbf{E}_{k} &= \mathbf{E}_{k-1} - (A^{T}A)^{-1} A^T\left(  \mathcal{G}(\mathbf{u}^{(k)})+A \bm\theta^{(k-1)}\right)\notag\\
             &= \mathbf{E}_{k-1}  - (A^{T}A)^{-1}A^T \left( \mathcal{G}(\mathbf{u}^{(k)})+A \bm\theta^{(k-1)} -  \mathcal{G}(\mathbf{u}^{\dagger})-A\bm\theta^{\dagger} \right),
\end{align*}
where we used $ G(\mathbf{u}^{\dagger};\bm\theta^{\dagger})=\mathcal{G}(\mathbf{u}^{\dagger})+A \bm\theta^{\dagger}=0$. 
Therefore, we conclude that
\begin{equation}\label{eq33}
    \mathbf{E}_{k} = -(A^{T}A)^{-1} A^T \left( \mathcal{G}(\mathbf{u}^{(k)})-\mathcal{G}(\mathbf{u}^{\dagger})\right).
\end{equation}
On the other hand, using Equation (\ref{eq21}) in the definition of the state error $\mathbf{e}_{k}=\mathbf{u}^{k}-\mathbf{u}^{\dagger}$ leads to
\begin{align*}
\mathbf{e}_{k+1} = \mathbf{e}_{k} - \left( \mathcal{G}'^{T} \mathcal{G}'+\alpha H^T H\right)^{-1} \left( \mathcal{G}'^{T} \left(\mathcal{G}(\mathbf{u}^{(k)})+  A \bm\theta^{(k)}\right) +\alpha H^T  (H\mathbf{u}^{(k)}-\mathbf{y})\right),
\end{align*}
where $\mathcal{G}'$ and $\mathcal{G}'^{T}$ are computed at $\mathbf{u}^{(k)}$. By substituting Equation (\ref{eq2}) with $\eta_j=0, \ \forall  j=0,\cdots,N$ into the expression above, we obtain the following
\begin{eqnarray*}
\mathbf{e}_{k+1} &=& \mathbf{e}_k - \left( \mathcal{G}'^{T} \mathcal{G}'+\alpha H^T H\right)^{-1} \left( \mathcal{G}'^{T} \left(\mathcal{G}(\mathbf{u}^{(k)})+A \bm\theta^{(k)}\right) +\alpha H^T H \mathbf{e}_k  \right)\\\nonumber
&=& \mathbf{e}_k - \left( \mathcal{G}'^{T} \mathcal{G}'+\alpha H^T H\right)^{-1} \left( \mathcal{G}'^{T} \left(\mathcal{G}(\mathbf{u}^{(k)})+A \bm\theta^{(k)} - \mathcal{G}(\mathbf{u}^{\dagger}) - A\bm\theta^{\dagger}\right) +\alpha H^T H \mathbf{e}_k  \right),\nonumber
\end{eqnarray*}
where we used $\mathcal{G}(\mathbf{u}^{\dagger})+A\bm\theta^{\dagger}=0$.
Opening the brackets and using the following property 
\begin{eqnarray*}
I-(\mathcal{G}^{'T}\mathcal{G}'+\alpha H^T H)^{-1}\alpha H^TH = (\mathcal{G}^{'T}\mathcal{G}'+\alpha H^TH)^{-1}\mathcal{G}^{'T}\mathcal{G}',
\end{eqnarray*}
we get
\begin{equation*}\label{eq_37}
    \mathbf{e}_{k+1} = \left( \mathcal{G}'^{T}\mathcal{G}'+\alpha H^T H\right)^{-1}\mathcal{G}'^{T}\left( \mathcal{G}' \mathbf{e}_k - \mathcal{G}(\mathbf{u}^{(k)})+\mathcal{G}(\mathbf{u}^\dagger) - A E_k\right).
\end{equation*}
Substituting Equation~\eqref{eq33} into the expression above leads to 
\begin{equation*}
	\mathbf{e}_{k+1} 
	=(\mathcal{G}^{'T}\mathcal{G}'+\alpha H^T H)^{-1} \mathcal{G}^{'T} 
 \left[
 \mathcal{G}' \mathbf{e}_k - \mathcal{G}(\mathbf{u}^{(k)})+\mathcal{G}(\mathbf{u}^\dagger) + A(A^{T}A)^{-1} A^T \left( \mathcal{G}(\mathbf{u}^{(k)})-\mathcal{G}(\mathbf{u}^{\dagger})\right)\right].
\end{equation*}
Taking norm of both sides, using assumptions \eqref{assth5_0}--\eqref{assth5_1} and Lemma \ref{lemma1}, we get
\begin{equation*}\label{eq_39}
    \|\mathbf{e}_{k+1}\| \le \frac{1}{2L_3c}\left(\frac{L_3}{2}\|\mathbf{e}_{k}\|^2 + L_0\|A(A^{T}A)^{-1} A^T\|\|\mathbf{e}_{k}\| \right),
\end{equation*}
which we rewrite into the following expression
\begin{equation*}\label{eq_39}
    \|\mathbf{e}_{k+1}\| \le \frac{1}{4c}\left(\|\mathbf{e}_{k}\| + b \right)^2\le\frac{1}{2c}\|\mathbf{e}_{k}\|^2 + \frac{1}{2c}b^2,
\end{equation*}
where $b = \|A(A^{T}A)^{-1} A^T\|L_0/L_3$. Next, we recursively obtain the following inequality
\[
    \|\mathbf{e}_{k}\|\le 2^{-k} \left( \frac{1}{c} \right)^{2^k-1}\|\mathbf{e}_{0}\|^{2^k} + \frac{c}{2} \sum_{i=1}^k \left(\frac{b}{c}\right)^{2^i},\quad \mbox{for}\quad k=1,2,\dots
\]
Since by assumption $b/c<1$ and $\|e_0\|<c$, we have that
\[
    \|\mathbf{e}_{k}\| < 2^{-k} c+ \frac{b}{2}\sum_{i=0}^k \left(\frac{b}{c}\right)^{i},
\]
and in the limit of $k$ goes to infinity 
\[
   \limsup_{k\to\infty} \|\mathbf{e}_{k}\| < \frac{b/2}{1-b/c}.
\]
Furthermore, taking norm of Equation~\eqref{eq33} leads to 
\begin{equation*}
    \|\mathbf{E}_{k}\|\le L_0\|(A^{T}A)^{-1} A^T\| \|\mathbf{e}_{k}\|.
\end{equation*}
By taking limit of both sides of the above inequality, as $k\to\infty$, we conclude that
\begin{align*}
\limsup_{k\to\infty}\|\mathbf{E}_k\| < L_0\|(A^{T}A)^{-1} A^T\|  \frac{b/2}{1-b/c}.
\end{align*}
\end{proof}
\end{appendix}
\bibliography{bibliography1}

\begin{thebibliography}{10}
\expandafter\ifx\csname url\endcsname\relax
  \def\url#1{\texttt{#1}}\fi
\expandafter\ifx\csname urlprefix\endcsname\relax\def\urlprefix{URL }\fi
\expandafter\ifx\csname href\endcsname\relax
  \def\href#1#2{#2} \def\path#1{#1}\fi

\bibitem{Jaz70}
A.~H. Jazwinski, Stochastic processes and filtering theory, Mathematics in
  science and engineering, Academic press, New York, 1970, uKM.

\bibitem{lewis1985use}
J.~M. Lewis, J.~C. Derber, The use of adjoint equations to solve a variational
  adjustment problem with advective constraints, Tellus A (1985).

\bibitem{talagrand1997assimilation}
O.~Talagrand, Assimilation of observations, an introduction (gtspecial
  issueltdata assimilation in meteology and oceanography: Theory and practice),
  Journal of the Meteorological Society of Japan. Ser. II (1997).

\bibitem{talagrand1987variational}
O.~Talagrand, P.~Courtier, Variational assimilation of meteorological
  observations with the adjoint vorticity equation. i: Theory, Quarterly
  Journal of the Royal Meteorological Society (1987).

\bibitem{sasaki1970some}
Y.~Sasaki, Some basic formalisms in numerical variational analysis, Monthly
  Weather Review (1970).

\bibitem{tr2006accounting}
Y.~Tr'emolet, Accounting for an imperfect model in 4d-var, Quarterly Journal of
  the Royal Meteorological Society: A journal of the atmospheric sciences,
  applied meteorology and physical oceanography 132~(621) (2006) 2483--2504.

\bibitem{tremolet2007model}
Y.~Tr{\'e}molet, Model-error estimation in 4d-var, Quarterly Journal of the
  Royal Meteorological Society: A journal of the atmospheric sciences, applied
  meteorology and physical oceanography 133~(626) (2007) 1267--1280.

\bibitem{johnson2005very}
C.~Johnson, N.~Nichols, B.~Hoskins, Very large inverse problems in atmosphere
  and ocean modelling, International journal for numerical methods in fluids
  47~(8-9) (2005) 759--771.

\bibitem{lorenc2000met}
A.~Lorenc, S.~Ballard, R.~Bell, N.~Ingleby, P.~Andrews, D.~Barker, J.~Bray,
  A.~Clayton, T.~Dalby, D.~Li, et~al., The met. office global three-dimensional
  variational data assimilation scheme, Quarterly Journal of the Royal
  Meteorological Society 126~(570) (2000) 2991--3012.

\bibitem{nichols2003data}
N.~Nichols, Data assimilation: Aims and basic concepts, in: Data Assimilation
  for the Earth System, Springer, 2003, pp. 9--20.

\bibitem{engl2000inverse}
H.~W. Engl, Inverse problems and their regularization, Computational
  mathematics driven by industrial problems (2000).

\bibitem{kaipio2006statistical}
J.~Kaipio, E.~Somersalo, Statistical and computational inverse problems,
  Springer Science \& Business Media, 2006.

\bibitem{brett2013accuracy}
C.~E. Brett, K.~F. Lam, K.~Law, D.~McCormick, M.~R. Scott, A.~Stuart, Accuracy
  and stability of filters for dissipative pdes, Physica D: Nonlinear Phenomena
  245~(1) (2013) 34--45.

\bibitem{hayden2011discrete}
K.~Hayden, E.~Olson, E.~S. Titi, Discrete data assimilation in the lorenz and
  2d navier--stokes equations, Physica D: Nonlinear Phenomena 240~(18) (2011)
  1416--1425.

\bibitem{MLPL2013}
A.~Moodey, A.~Lawless, R.~Potthast, P.~Van~Leeuwen, Nonlinear error dynamics
  for cycled data assimilation methods, Inverse Problems 29~(2) (2013) 25002.
\newblock \href
  {https://doi.org/https://dx.doi.org/10.1088/0266-5611/29/2/025002}
  {\path{doi:https://dx.doi.org/10.1088/0266-5611/29/2/025002}}.

\bibitem{gratton2007approximate}
S.~Gratton, A.~S. Lawless, N.~K. Nichols, Approximate gauss--newton methods for
  nonlinear least squares problems, SIAM Journal on Optimization 18~(1) (2007)
  106--132.

\bibitem{cartis2021convergent}
C.~Cartis, M.~H. Kaouri, A.~S. Lawless, N.~K. Nichols, Convergent least-squares
  optimisation methods for variational data assimilation, arXiv preprint
  arXiv:2107.12361 (2021).

\bibitem{courtier1994strategy}
P.~Courtier, J.-N. Th{\'e}paut, A.~Hollingsworth, A strategy for operational
  implementation of 4d-var, using an incremental approach, Quarterly Journal of
  the Royal Meteorological Society 120~(519) (1994) 1367--1387.

\bibitem{lawless2005investigation}
A.~Lawless, S.~Gratton, N.~Nichols, An investigation of incremental 4d-var
  using non-tangent linear models, Quarterly Journal of the Royal
  Meteorological Society: A journal of the atmospheric sciences, applied
  meteorology and physical oceanography 131~(606) (2005) 459--476.

\bibitem{lawless2005approximate}
A.~Lawless, S.~Gratton, N.~Nichols, Approximate iterative methods for
  variational data assimilation, International journal for numerical methods in
  fluids 47~(10-11) (2005) 1129--1135.

\bibitem{dennis1996numerical}
J.~E. Dennis~Jr, R.~B. Schnabel, Numerical methods for unconstrained
  optimization and nonlinear equations, SIAM, 1996.

\bibitem{Lorenz63}
E.~N. Lorenz, {Deterministic Nonperiodic Flow}, Journal of Atmospheric Sciences
  20 (1963) 130--148.
\newblock \href {https://doi.org/10.1175/1520-0469(1963)020<0130:DNF>2.0.CO;2}
  {\path{doi:10.1175/1520-0469(1963)020<0130:DNF>2.0.CO;2}}.

\bibitem{lorenz1996predictability}
E.~N. Lorenz, Predictability - a problem partly solved, in: T.~Palmer,
  R.~Hagedorn (Eds.), Proceedings of seminar on Predictability, Vol.~1, ECMWF,
  Cambridge University Press, Reading, UK, 1996, pp. 1--18.

\end{thebibliography}

\end{document}